\documentclass{amsart}
\usepackage{graphics}
\usepackage{graphicx}

\usepackage{amsfonts,amssymb,amsopn,amscd,amsmath}
 \usepackage{mathrsfs}
\usepackage[usenames,dvipsnames]{xcolor}
 \newtheorem{thm}{Theorem}[section]
 \newtheorem{cor}[thm]{Corollary}
 \newtheorem{lem}[thm]{Lemma}{\rm}

 {\rm}
 \newtheorem{assumption}[thm]{Assumption}
 \newtheorem{rem}[thm]{Remark}
 
\numberwithin{equation}{section}

\begin{document}

\def\red{\color{red}}
\def\bl{\color{blue}}
\def\ora{\color{orange}}
\def\green{\color{green}}
\def\br{\color{brown}}

\def\la{\langle}
\def\ra{\rangle}
\def\e{{\rm e}}
\def\x{\mathbf{x}}
\def\by{\mathbf{y}}
\def\bz{\mathbf{z}}
\def\F{\mathcal{F}}
\def\R{\mathbb{R}}
\def\T{\mathbf{T}}
\def\N{\mathbb{N}}
\def\K{\mathbf{K}}
\def\bK{\overline{\mathbf{K}}}
\def\Q{\mathbf{Q}}
\def\M{\mathbf{M}}
\def\O{\mathbf{O}}
\def\C{\mathbf{C}}
\def\P{\mathbf{P}}
\def\Z{\mathbb{Z}}
\def\H{\mathcal{H}}
\def\A{\mathbf{A}}
\def\V{\mathbf{V}}
\def\AA{\overline{\mathbf{A}}}
\def\B{\mathbf{B}}
\def\c{\mathbf{C}}
\def\L{\mathbf{L}}
\def\bS{\mathbf{S}}
\def\H{\mathcal{H}}
\def\I{\mathbf{I}}
\def\Y{\mathbf{Y}}
\def\X{\mathbf{X}}
\def\G{\mathbf{G}}
\def\B{\mathbf{B}}
\def\f{\mathbf{f}}
\def\z{\mathbf{z}}
\def\y{\mathbf{y}}
\def\d{\hat{d}}
\def\bx{\mathbf{x}}
\def\y{\mathbf{y}}
\def\h{\mathbf{h}}
\def\g{\mathbf{g}}
\def\v{\mathbf{v}}
\def\tv{\tilde{\mathbf{v}}}
\def\g{\mathbf{g}}
\def\tg{\tilde{\mathbf{g}}}
\def\w{\mathbf{w}}
\def\b{\mathcal{B}}
\def\a{\mathbf{a}}
\def\q{\mathbf{q}}
\def\u{\mathbf{u}}
\def\s{\mathcal{S}}
\def\cc{\mathcal{C}}
\def\co{{\rm co}\,}
\def\cp{{\rm CP}}
\def\tx{\tilde{\x}}
\def\bP{\mathbb{P}}
\def\supmu{{\rm supp}\,\mu}
\def\supnu{{\rm supp}\,\nu}
\def\m{\mathcal{M}}
\def\bR{\mathbf{R}}
\def\om{\mathbf{\Omega}}
\def\c{\mathbf{c}}
\def\s{\mathcal{S}}
\def\k{\mathcal{K}}
\def\la{\langle}
\def\ra{\rangle}
\def\blambda{{\boldsymbol{\lambda}}}
\def\balpha{{\boldsymbol{\alpha}}}
\def\bbeta{{\boldsymbol{\beta}}}
\def\bphi{{\boldsymbol{\phi}}}
\def\bmu{{\boldsymbol{\mu}}}
\def\bnu{{\boldsymbol{\nu}}}
\def\HM{\mathbf{HM}}
\def\tHM{\widetilde{\mathbf{HM}}}
\def\tbmu{\tilde{\bmu}}

\title[Variational properties of integral discriminants]{Gaussian-like fixed point and variational properties of integral discriminants}

\author{Jean-Bernard Lasserre}
\thanks{The author is supported by the AI Interdisciplinary Institute ANITI  funding through the %french program ``Investing for the Future PI3A" under the grant agreement number 
ANITI AI Cluster program under the Grant agreement ANR-23-IACL-0002, and also by the Marie-Sklodovska-Curie european doctoral network TENORS, grant 101120296.
This research is also part of the programme DesCartes and is supported by the National Research Foundation, Prime Minister's Office, Singapore under its Campus for Research Excellence and Technological Enterprise (CREATE) programme.}
\address{Jean-Bernard Lasserre: LAAS-CNRS and Toulouse- School of Economics (TSE)\\
University of Toulouse\\
LAAS, 7 avenue du Colonel Roche\\
31077 Toulouse C\'edex 4,France}
\email{lasserre@laas.fr}
\date{}

\begin{abstract}
We consider partition functions $Z(g)=\int \exp(-g(\x))d\x$
where $g$ is a nonnegative polynomial action
(a degree-$2n$ form) vanishing only at
the origin. Such integrals, known as integral discriminants, appear in statistical
mechanics, quantum field theory, and the theory of exponential families.
We show that the associated Boltzmann measure $d\mu=\exp(-g(\x))d\x$ satisfies a fixed-point
property identity relating in a simple manner its degree-$2n$ moments to the coefficients of  $g$.
This generalizes familiar identities for the exponential distribution (degree-$1$) on the positive orthant
and the Gaussian measure (degree-$2$). 
We further show that $g$ is characterized by three variational principles, including 
a maximum-entropy principle under scaled moments constraints, extending
the Gaussian extremality principle to arbitrary even-degree homogeneous actions.
Exploiting these identities in a truncated-moment numerical scheme (known as 
the Moment-SOS hierarchy), strengthens the standard semidefinite relaxations, and
results in a much faster convergence, thus allowing more efficient approximations of the partition function
$Z(g)$ as well as moments of $\mu$. 
\end{abstract}

\maketitle

\section{Introduction}

Integrals of the form $Z(g)=\int_{\R^d}\exp(g(\x))\,d\x$ 
where $g$ is a nonnegative form, arise naturally in statistical mechanics and
quantum field theory as partition functions of models with polynomial actions. In the
terminology of Dolotin, Morozov and Shakirov, they are referred to as \emph{integral
discriminants} \cite{morozov-1,morozov-2,stoyanovsky-1,stoyanovsky-2}.
Despite their elementary appearance, such integrals remain poorly
understood: Indeed although $Z(g)$ is known to be a generalized hypergeometric function of the
coefficients of $g$, explicit closed-form expressions are available only in a few very special
cases \cite{morozov-1,morozov-2}.
The above approach emphasizes the algebraic and differential structure as
a function of the couplings appearing in $g$, as a means to obtain $Z(g)$ explicitly. In contrast, the present work adopts a
complementary viewpoint and focuses on structural relations between the polynomial
action and the associated Boltzmann weight 
\[d\mu(\x)\,=\,\exp(-g(\x))\,d\x\,,\]
to provide additional characterizations of the couple $(\mu,g)$.
We restrict attention to degree-$2n$ nonnegative forms that
vanish only at the origin. Under this assumption, the integral $Z(g)$ is finite and the sublevel set
\[G\,:=\,\{\,\x\in\R^d:\: g(\x)\,\leq\,1\,\}\]
is compact.
An important property of homogeneous actions is the exact identity
\begin{equation}
\label{crucial}
\int_{\R^d}f(\x)\,\exp(-g(\x))\,d\x\,=\,\Gamma(1+\frac{d+s}{2n})\int_Gf(\x)\,d\x\,,\end{equation}
valid for any continuous and positively homogeneous observable $f$ of degree $s$. This relation
connects the Boltzmann measure with the uniform measure on $G$, and allows one to
translate moment identities for $\mu$ into geometric properties of the sublevel set $G$ of $g$. 
For instance, it has
already proved useful for studying convexity of $Z(g)$ and generalized L\"owner-John
problems \cite{lowner}, inverse moment questions \cite{recovery}, and numerical approximation of 
$Z(g)$ \cite{siaga}.

The purpose of this paper is to show that the pair $(\mu,g)$ satisfies identities that closely
parallel well-known Gaussian properties. For a Gaussian action, the covariance matrix
satisfies both a fixed-point identity and a maximum-entropy principle under second 
moment constraints. We show that analogous statements hold for arbitrary (nonnegative) even-degree
homogeneous actions.

\subsection*{Main results}
\subsection*{1. Fixed-point structure.} 
We derive a nonlinear fixed-point equation for the degree-$2n$ moments of $\mu$, in
which the same moments appear linearly in the exponent of the Boltzmann weight.
Equivalently, the coefficients of $g$ are solution of a simple linear system
involving the degree-$2n$
moment matrix of $\mu$. This provides a direct higher-degree analogue of the Gaussian
identity relating the covariance matrix to quadratic expectations.  Those linear relations 
had been already obtained in two ways: 
via Ward's identities specialized to the context of integral discriminants as in \cite{morozov-1,morozov-2}, or
via Stokes theorem applied to integration w.r.t. to Lebesgue
measure on $G$, and then combined with \eqref{crucial} as done in \cite{recovery,lowner}.
In \cite{morozov-1,morozov-2} the Ward identities are seen as a useful intermediate tool for the ultimate
goal of obtaining $Z(g)$ in closed form.  
In this work we follow a different route. Rather than viewing the Ward identities merely as constraints on moments, we reorganize a subset of them into a linear system for the coefficients of $g$
as already observed in \cite{recovery}. The novelty of the present work is to show that this system can be interpreted as a fixed-point characterization of $g$
in terms of moments of orders $2n$ and $4n$,
paralleling and extending the classical (degree-$2$) Gaussian case. Remarkably, the form $g$
 is reconstructed from finitely many statistics generated by its own Boltzmann measure. 
This viewpoint also leads naturally to variational characterizations of the pair 
$(\mu,g)$ (see below).  To the best of our knowledge, this fixed-point and variational viewpoint has not been presented before.
 
\subsection*{2. Variational principles.}
We show that $g$ is uniquely characterized by three variational properties:

-- (i) it is proportional to the best $L^2(\mu)$-approximation 
of the constant function by a degree-$2n$ form.

-- (ii) it minimizes the $L^2(\mu)$-norm among all degree-$2n$ forms 
with same $L^1(\mu)$-norm as  $g$.

-- (iii) the normalized density $p^*(\x)\propto \exp(-g(\x))$ uniquely solves
a maximum-entropy problem under 
scaled degree-$2n$ moment constraints, a higher degree  analogue of the Gaussian entropy
maximization principle.

\subsection*{3. Numerical implications.}
Since
\[Z(g)\,=\,\Gamma(1+\frac{d}{2n})\,\mathrm{vol}(G)\,,\]
computing $Z(g)$ reduces to computing the volume of the basic semi-algebraic set $G$.
As in turn the latter can be  viewed as an instance of the Generalized Moment Problem (GMP),
one may apply the so-called the Moment-SOS hierarchy
devoted to solving the GMP \cite{acta}. It consists of solving a hierarchy of semidefinite relaxations 
of increasing size. The resulting  associated sequence of optimal values is monotone non increasing
and converges to $\mathrm{vol}(G)$ from above.
In the present context of homogeneous forms, the fixed-point
identities alluded to above, imply that additional linear moment constraints (essentially Stokes theorem in disguise) can be enforced in the semidefinite relaxations\footnote{A semidefinite program is a conic convex optimization problem 
which up arbitrary precision fixed in advance,  can be solved efficiently; for more details the interested reader is referred to \cite{anjos}.} (because they are satisfied by the optimal measure $1_Gd\x$). Remarkably,
the inclusion of these additional Stokes moment constraints results in a dramatic acceleration of the convergence.

So these results provide a unified perspective on integral discriminants, 
highlighting higher-degree analogues of Gaussian extremality and fixed-point structures. They also underpin efficient numerical schemes for approximating partition functions and moments via truncated moment hierarchies.
\section{Main result}

\subsection{Notation, definitions and preliminary result}
Let $\R[\x]=\R[x_1,\ldots,x_d]$ be the ring of polynomials in the variables
$\x=(x_1,\ldots,x_d)$.
For $\balpha\in\N^d$ let
$\vert\balpha\vert:=\sum_i\alpha_i$, $\N^d_n=\{\balpha\in\N^d: \vert\balpha\vert=n\}$, and let $\v_n(\x)=(\x^{\balpha})_{\vert\balpha\vert=n}$
be the vector of all degree-$n$ monomials. 
Denote by $\H[\x]_n\subset\R[\x]$ the vector space of 
degree-$n$ forms, with usual monomial basis 
$\v_n(\x)$. For a real symmetric matrix $\A$,
the notation $\A\succeq0$ (resp. $\A\succ0$) stands for $\A$ is positive semidefinite
(resp. positive definite). A function $f:\R^d\to\R$ is positively homogeneous of degree $s\in\R$, if 
\[f(\lambda\,\x)\,=\,\lambda^s\,f(\x)\,,\quad\forall \x\in\R^d\,,\:\forall\lambda>0\,.\]

For every integer $n$,
let $s_n:=\mathrm{dim}(\mathcal{H}[\x]_n)={d-1+n\choose n}$ and for
$g\in\H[\x]_{n}$, let $\g\in\R^{s_n}$ be its vector of coefficients in
the monomial basis $\v_n(\x)$. That is:
\[\x\mapsto g(\x)\,=\,\v_n(\x)^T\g\,,\quad\forall \x\in\R^d\,.\]

For a finite Borel measure on $\R^d$ denote by 
\[\HM_{n}(\mu)\,:=\,\int \v_n(\x)\v_n(\x)^T\,d\mu(\x)\,,\]
the degree-$n$ (homogeneous) moment matrix associated with $\mu$. If
$\HM_n(\mu)\succ0$ then there exists 
a family $(P_\balpha)_{\balpha\in\N^d_n}\subset \mathcal{H}_n[\x]$
of degree-$n$ forms that are orthonormal w.r.t. $\mu$, that is,
\begin{equation}
 \int P_\balpha\,P_\bbeta\,d\mu\,=\,\delta_{\balpha=\bbeta}\,,\quad\forall \balpha,\bbeta\in\N^d_n\,.
\end{equation}
\subsection*{Christoffel-Darboux kernel}With $\HM_n(\mu)\succ0$, introduce the kernel
\[(\x,\y)\mapsto K^\mu_n(\x,\y)\,:=\,\sum_{\balpha\in\N^d_n}P_\balpha(\x)\,P_\balpha(\y)\,,\quad\forall \x,\y\in\R^d\,,\]
which is called the \emph{Christoffel-Darboux kernel} (CD-kernel)  associated with $\mu$. The kernel
$K^\mu_n$ has the reproducing property,
\[h(\x)\,=\,\int K^\mu_n(\x,\y)\,h(\y)\,d\mu(\y)\,,\quad\forall h\in\mathcal{H}_n[\x]\,,\]
and so with the scalar product $\langle g,h\ra=\int fg\,d\mu$,  
$(\mathcal{H}_n[\x],\la\cdot,\cdot\ra)$ is a reproducing kernel Hilbert space (RKHS).

Throughout the rest of the paper we consider the following assumption.
\begin{assumption}
\label{ass-1}
$g\in\H[\x]_{2n}$ is a nonnegative degree-$2n$ form such that $g(\x)=0$ only if $\x=0$.
(Then the sublevel set $G:=\{\x\in\R^d:g(\x)\,\leq\,1\}$ is compact.)
\end{assumption}

With $g\in\H[\x]_{2n}$ satisfying Assumption
\ref{ass-1}, let $\lambda$ be the Lebesgue measure on $G:=\{\x: g(\x)\leq 1\}$, and let
$\blambda^{(2n)}=(\lambda_{\balpha})_{\balpha\in\N^d_{2n}}$ be its associated vector of degree-$2n$
moments, while $\HM_{2n}(\lambda)$ is its  degree-$2n$ moment matrix.
Let $\mu$ be the finite Borel measure on $\R^d$ with density
$\exp(-g(\x))$ w.r.t. Lebesgue measure on $\R^d$, i.e.,
\[\mu(B)\,=\,\int_B\exp(-g(\x))\,d\x\,,\quad\forall B\in\mathcal{B}(\R^d)\,,\]
and let $\bmu^{(2n)}=(\mu_{\balpha})_{\vert\balpha\vert=2n}$ be its vector of 
degree-$2n$ moments. We first recall a result stated in e.g. \cite{lowner} and \cite{morozov-2}.
\begin{thm}(\cite{morozov-2})
\label{thm1}
Let Assumption \ref{ass-1} hold. Then for every $\balpha\in\N^d$:
\begin{equation}
\label{thm1-1}
\int_G \x^{\balpha}\,d\x\,=\,\frac{1}{\Gamma(1+(d+\vert\balpha\vert)/(2n))}\int_{\R^d}\x^{\balpha}\exp(-g(\x))\,d\x\,,
\end{equation}
If $f$ is a continuous and positively homogeneous function of degree $s$,  one also has
\begin{equation}
\label{thm1-f}
\int_G f(\x)\,d\x\,=\,\frac{1}{\Gamma(1+(d+s)/(2n))}\int_{\R^d}f(\x)\exp(-g(\x))\,d\x\,.
\end{equation}
\end{thm}
Notice that in particular, as a  consequence of Theorem \ref{thm1},
\begin{eqnarray}
\nonumber
\HM_{2n}(\lambda)&=&\int_G \v_{2n}(\x)\v_{2n}(\x)\,d\x\\
\nonumber
%\label{lem1-2}
&=&\frac{1}{\Gamma(1+(d+4n)/(2n))}\int_{\R^d}\v_{2n}(\x)\v_{2n}(\x)^T\exp(-g(\x))\,d\x\\
\label{thm1-2}
&=&\frac{1}{\Gamma(1+(d+4n)/(2n))}\HM_{2n}(\mu)\,.
\end{eqnarray}

\subsection{A Gaussian-like fixed point property}

In order to present the result in a very compact form, and given the finite Borel measure 
$\mu=\exp(-g(\x))d\x$ with $0<g\in\mathcal{H}[\x]_{2n}$, one first 
considers a basis $\{P_\balpha: \balpha\in\N^d_{2n}\}$ of $\mathcal{H}[\x]_{2n}$ with polynomials
that are orthonormal w.r.t. $\mu$, i.e.,
\[\int P_\balpha\,P_\bbeta\,d\mu\,=\,\delta_{\balpha=\bbeta}\,,\quad\forall\balpha,\bbeta\in\R^d_{2n}\,.\]
\subsection*{In an orthonormal basis $(P_\alpha)\perp\mu$} 
(As $\HM_{2n}(\mu)\succ0$, such an orthonormal family exists and can be obtained in various ways.)
Next, let $\bP_{2n}(\x):=\left(P_\balpha(\x)\right)_{\balpha\in\N^d_{2n}}$, so that
$g(\x)=\la\tg,\bP_{2n}(\x)\ra$ where $\tg$
is the vector of  coefficients of $g$ in the basis $\bP_{2n}(\x)$; recall that in contrast to the canonical basis 
$\v_n(\x)$, $\bP_{2n}(\x)$ \emph{depends} on $g$. Next, let
\[\tbmu^{(2n)}\,=\,\left(\int P_\balpha\,d\mu\right)_{\balpha\in\N^d_{2n}}\,=\,\int \bP_{2n}(\x)\,d\mu\]
be the associated vector  of degree-$2n$ moments. So expressed in the basis $\bP_{2n}(\x)$, the
homogeneous moment matrix $\tHM_{2n}(\mu)$ is simply the identity matrix $\I$ because
\[\tHM_{2n}(\mu)\,=\,\int\bP_{2n}(\x)\bP_{2n}(\x)^Td\mu\,=\,\left(\int P_\balpha\,P_\bbeta)\,d\mu\right)_{\balpha,\bbeta\in\N^d_{2n}}\,=\,\mathbf{I}\,.\]
For every integer $n$, let $c_{n}:=1+d/n$. Our first result states the following
\begin{thm}
\label{intro-th1}
With $0<g\in\mathcal{H}[\x]_{2n}$ and $\mu=\exp(-g(\x))d\x$: 
\begin{eqnarray}
\label{intro-th21-0}
\tg&=&c_{2n}\,\tbmu^{(2n)}\,=\,c_{2n}\,\int \bP_{2n}(\x)\,\exp(-\la\bP_{2n}(\x),\tg\ra)\,d\x\,.\\
\label{intro-th21-1}
g(\x)&=&c_{2n}\,\la\bP_{2n}(\x),\tbmu^{(2n)}\ra\,=\,c_{2n}\,\int K^\mu_{2n}(\x,\y)\,d\mu(\y),\quad\forall \x\in\R^d\,\\
 \label{intro-th21-2}
 \tbmu^{(2n)}&=&\int \bP_{2n}(\x)\,\exp(-c_{2n}\,\la\bP_{2n}(\x),\tbmu^{(2n)}\ra)\,d\x\,.
 %label{intro-th21-3} \g
 \end{eqnarray}
\end{thm}
\begin{proof}
 By \cite[Lemma 1, p. 675]{recovery}, 
\begin{equation}
\label{aux-0}
\int_G \x^\balpha\,\,g(\x)\,d\x\,=\,\frac{d+\vert\balpha\vert}{d+2n+\vert\balpha\vert}\int_G \x^\balpha\,d\x\,,\quad
\forall\balpha\in \N^d\,.\end{equation}
Combining with \eqref{thm1-1}, and letting $c_{2n}:=1+d/(2n)$, one obtains
\begin{equation}
\label{aux-1}
\int_{\R^d} \x^\balpha\,\,g(\x)\,\exp(-g(\x))\,d\x\,=\,c_{2n}\,\int_{\R^d} \x^\balpha\,\exp(-g(\x))\,d\x\,,\quad
\forall\balpha\in \N^d_{2n}\,,\end{equation}
where one has repeated used $\Gamma(1+t)=t\Gamma(t)$.
Next, observe that
\begin{eqnarray}
\label{aux-2}
\widetilde{\HM}_{2n}(\mu)\,\tg&=&\int_{\R^d} \left(P_\balpha(\x)\right)_{\balpha\in\N^d_{2n}}\,g(\x)\,\exp(-g(\x))\,d\x\\
\nonumber
&=&c_{2n}\,\int_{\R^d} \left(P_\balpha(\x)\right)_{\balpha\in\N^d_{2n}}\exp(-g(\x))\,d\x\quad\mbox{[by \eqref{aux-1}]}\\
\nonumber
&=&c_{2n}\,\tilde{\bmu}^{(2n)}\,,
\end{eqnarray}
which is \eqref{intro-th21-0} because $\widetilde{\HM}_{2n}(\mu)=\mathbf{I}$.
 \end{proof}
Theorem \eqref{intro-th1} states that 
$\mu$ satisfies a remarkable ``fixed-point" property in 
a quite simple compact form, as for Gaussian measures.  
Namely, the degree-$2n$ moment vector $\tbmu^{(2n)}$ in
the left-hand-side of \eqref{intro-th21-2} also appears linearly in
the argument of the exponential density in the right-hand-side.  
Similarly, the vector of coefficients $\tg$ in
the left-hand-side of \eqref{intro-th21-0} also appears linearly in
the argument of the exponential density in the right-hand-side.  
Moreover
the degree-$2n$ moments $\tbmu^{(2n)}$ encode the form $g$ by providing its vector of coefficients. 

In particular, and as a consequence:
\begin{cor}
\label{intro-cor1}
With $0<g\in\mathcal{H}[\x]_{2n}$ and $\mu=\exp(-g(\x))d\x$, 
let
\[g\mapsto Z(g)\,:=\,\int_{\R^d} \exp(-g(\x))\,d\x\]
be the partition function associated with $\mu$. Then
\begin{equation}
\label{intro-cor1-1}
Z(g)\,=\,\frac{4n^2}{d\,(d+2n)}\Vert \tg\Vert^2\,=\,.\frac{d+2n}{d}\Vert \tbmu^{(2n)}\Vert^2\,.
\end{equation}
\end{cor}
\begin{proof}
 By \eqref{intro-th21-0} 
 \[c_{2n}\Vert\tilde{\bmu}^{(2n)}\Vert^2\,=\,\la\tg,\tilde{\bmu}^{(2n)}\ra\,=\,\la\tg,\int \bP_{2n}(\x)\,d\mu\ra\,=\,\int g\,d\mu\,=\,\frac{d}{2n}\,Z(g)\,.\]
 Finally, $\Vert\tg\Vert^2=c_{2n}^2\Vert\tilde{\bmu}^{(2n)}\Vert^2$ implies 
 $\Vert\tg\Vert^2=Z(g)d(d+2n)/(4n^2)$.
\end{proof}
The expression of $Z(g)$ in \eqref{intro-cor1-1} is particularly compact and simple, but
of course the vector $\tg$ in the basis $\bP_{2n}(\x)$ is \emph{not} directly available for 
the input polynomial $g\in \mathcal{H}[\x]_{2n}$. 
Therefore we next provide the analogues of Theorem \ref{intro-th1} and Corollary \ref{intro-cor1} 
when using
the usual monomial basis $\v_{2n}(\x)=(\x^\balpha)_{\balpha\in\N^d_{2n}}$, which in contrast to $\bP_{2n}(\x)$,
does \emph{not} depend on $g$. They  now involve the inverse $\HM_{2n}(\mu)^{-1}$ of the moment matrix (in the monomial basis). 
\subsection*{In the canonical basis} 
Namely, with now $\g$ (resp. $\HM_{2n}(\mu)$) the vector of $g$ (resp. moment matrix of $\mu$)
in the canonical basis $\v_{2n}(\x)$, \eqref{aux-2} translates to
%\begin{equation}
% \HM_{2n}\,\g\,=\,\bmu^{(2n)}\,,
%\end{equation}
%from which one obtains
\begin{eqnarray}
\label{intro-monom-1}
\g&=&c_{2n}\,
\HM_{2n}(\mu)^{-1}\,\bmu^{(2)}\\
\label{intro-monom-11}
\bmu^{(2n)}&=&\int \v_{2n}(\x)\,\exp(-c_{2n}\,\la\v_{2n}(\x),\HM_2(\mu)^{-1}\bmu^{(2)}\ra)\,d\x\,.
\end{eqnarray}
The proof mimics that of Theorem \ref{intro-th1} and one obtains \eqref{aux-2} with now $\HM_{2n}(\mu)$
instead of $\widetilde{\HM}_{2n}(\mu)$, whence $\HM_{2n}(\mu)^{-1}$ in \eqref{intro-monom-1}
and in \eqref{intro-monom-11}.

In particular, in the degree-$2$ (Gaussian) case (i.e., with $2n=2$) with $g(\x)=\x^T\Sigma\x)$,
let $\nu:=\mu/\int \exp(-g(\x))d\x$ be the Gaussian probability measure. Then 
is expressed in terms of the degree-$4$ moments (in $\HM_2(\nu)$) and the degree-$2$ moments $\bnu^{(2)}$ being the vectorization $\mathrm{vec}(\Sigma)$ 
of $\Sigma$, so that
\[\bnu^{(2)}\:(=\,\mathrm{vec}(\Sigma))\,=\,\frac{
\int_{\R^d} \v_2(\x)\exp(-c_2\,\la\v_2(\x),\HM_2(\nu)^{-1}\bnu^{(2)}\ra)\,d\x}{\int_{\R^d}\exp(-g(\x))d\x}\,\]
%&=&\int \v_2(\x)\exp(-\x^T\Sigma\x)\,d\x\,,
%\end{eqnarray*}
which is the vectorization of the well-known (Gaussian) identity
\[\Sigma\,=\,\frac{\int \v_1(\x)\v_1(\x)^T\exp(-\v_1(\x)^T\Sigma^{-1}\v_1(\x))\,d\x}
{\int \exp(-g(\x))d\x}\,.\]
Similarly for the exponential measure $d\mu=\exp(-\blambda^T\x)\,d\x$ on $\R^d_+$ with $0<\blambda\in\R^d_+$, one obtains
\[\bmu^{(1)}\,=\,\int \v_1(\x)\exp(-c_1\,\la\v_1(\x),\HM_1(\mu)^{-1}\bmu^{(1)}\ra)\,d\x\,\]
\begin{rem}
\label{rem-stokes}
 Eq. \eqref{intro-monom-1}, or $\HM_{2n}(\mu)\,\g=c_{2n}\bmu^{(2n)}$, follows 
 from the  crucial identity \eqref{aux-0} from \cite[Lemma 1, p. 675]{recovery}, which in turn 
 is obtained from Stokes theorem (with vector field $\x$) by:
 \[\int_G \mathrm{Div}(\x\cdot\x^\balpha\,(g(\x)-1))\,d\x\,=\,\int_{\partial G}\frac{\la\nabla g(\x),\x\ra}{\Vert\nabla g(\x)\Vert} \x^\balpha\,(g(\x)-1)\,d\sigma(\x)\,=\,0\,,\]
 for every $\balpha\in\N^d$. 
 \end{rem}
\subsection{Variational properties}
One next provides three variational formulations associated with $g$ and $\mu$. Let
$\Vert \cdot\Vert_1$ (resp. $\Vert\cdot\Vert_2$) be the $L^1(\mu)$-norm (resp. $L^1(\mu)$-norm).
\begin{cor}
\label{intro-cor3}
With $0<g\in\mathcal{H}[\x]_{2n}$ and $\mu=\exp(-g(\x))d\x$:
\begin{eqnarray}
\label{intro-cor3-1}
g&=&c_{2n}\,\arg\min_{q\in\H_{2n}[\x]}\Vert q-1\Vert^2_2\\
\label{intro-cor3-2}
g&=&\arg\min_{q\in\H_{2n}[\x]}\{\,\Vert q\Vert_2\,:\: \Vert q\Vert_1\,=\,\Vert g\Vert_1\,\}
\end{eqnarray}
\end{cor}
\begin{proof}
With $q(\x)=\la\q,\v_{2n}(\x)\ra$ and as $\mu(\R^d)$ is a constant,  the optimization problem \eqref{intro-cor3-1} reads
\[\min_{q\in\mathcal{H}[\x]_{2n}}\int q^2d\mu -2\int q\,d\mu\,=\,\min_{\q}\,\q^T\HM_{2n}(\mu)\q-2\la \q,\bmu^{(2n)}\ra\,,\]
which is a convex optimization problem with strictly convex cost. Hence the unique critical point $\q^*$ satisfies
$\HM_{2n}(\mu)\,\q^*=\bmu^{(2n)}$, and therefore in view of \eqref{intro-monom-1},
$q^*=g/c_{2n}$, which yields the desired result \eqref{intro-cor3-1}.

Next, consider problem \eqref{intro-cor3-2} and observe that since 
$0<g\in\mathcal{H}[\x]_{2n}$, and $\vert h\vert$ is continuous and positively homogeneous of degree
$2n$, 
\begin{eqnarray*}
\int_{\R^d} \vert q\vert \,g\,\exp(-g)\,d\x&=&\frac{d+2n}{2n}\int_{\R^d}\vert q\vert\,\exp(-g)\,d\x
\,=\,\frac{d+2n}{2n}\Vert q\Vert_1\\
\Vert g\Vert^2_2\,=\,\int g^2\,\exp(-g)\,d\x&=&\frac{d+2n}{2n}\int g\,\exp(-g)\,d\x\,=\,\frac{d+2n}{2n}\Vert g\Vert_1\,,
\end{eqnarray*}
where we have used $0<g$ to obtain $\int g\,d\mu=\Vert g\Vert_1$.
Hence  as  $\Vert q\Vert_1=\Vert g\Vert_1$,
\begin{eqnarray*}
0\,\leq\, \int (\vert q\vert -g)^2\,d\mu&=&\Vert q\Vert_2^2-2\,\int \vert q\vert \,g\,d\mu+\Vert g\Vert^2_2\\
 &=&\Vert q\Vert_2^2-2\,\frac{d+2n}{2n}\Vert q\Vert_1+\frac{d+2n}{2n}\Vert g\Vert_1\\
  &=&\Vert q\Vert_2^2-\,\frac{d+2n}{2n}\Vert g\Vert_1\,=\,\Vert q\Vert_2^2-\,\Vert g\Vert^2_2\,,
  \end{eqnarray*}
and therefore 
  $\Vert q\Vert_2^2\geq\Vert g\Vert_2^2$, which in turn implies the desired result \eqref{intro-cor3-2}.
 \end{proof} 
 So $g$ is proportional to the best $L^2(\mu)$-approximation of the constant polynomial
 equal to $1$, by a degree-$2n$ form, and in addition, $g$ minimizes the $L^2(\mu)$ norm
 over all degree-$2n$ forms with same  $L^1(\mu)$-norm as $g$.
 
 Next one shows that $g$ is also the unique optimal solution of a max-entropy problem. Let 
 $L^1$ be the space of integrable functions w.r.t. Lebesgue measure.
  \begin{lem}
 Under Assumption \ref{ass-1}, let  $d\mu=\exp(-g)d\x$, $1/c^*:=\int \exp(-g)\,d\x$, and consider the 
 Max-Entropy problem:
 \begin{equation}
 \label{maxent}
 \begin{array}{rl}
 \tau\,=\displaystyle\min_{0<q\in L^1}&\{\,\displaystyle\int_{\R^d} q(\x)\,\ln q(\x)\,d\x: \:\displaystyle\int_{\R^d} q(\x)\,d\x=1\,;\\
 &\displaystyle\int_{\R^d} \v_{2n}(\x)\,q(\x)\,d\x\,=\,c^*\,\bmu^{(2n)}\,\}\,.\end{array}
 \end{equation}
 Then $q^*:=c^*\exp(-g)$ is the unique optimal solution of \eqref{maxent}.
 \end{lem}
 \begin{proof}
 Define the log-partition function
 \[\blambda\mapsto \Psi(\blambda)\,:=\,\left\{\begin{array}{rl}
 \ln \int \exp(\la\blambda,\v_{2n}(\x)\ra)\,d\x&\mbox{if the integral is finite}\\
 +\infty&\mbox{otherwise,}\end{array}\right.\]
  and its associated Legendre-Fenchel transform
 \[\Psi^*(\y)\,=\,\sup_{\blambda}\la \blambda,\y\ra -\Psi(\blambda)\,,\,\quad\forall \y\,,\]
  Then the optimization problem
  \begin{equation}
 \label{maxent-dual}
 \Psi^*(c^*\bmu^{(2n)})\,=\,\sup_{\blambda}\:\la \blambda,c^*\bmu^{(2n)}\ra -\Psi(\blambda)\,,\end{equation}
 is a \emph{dual} problem associated with the primal problem \eqref{maxent}.  Indeed let us show that \emph{weak duality} holds. Let $0<q\in L^1$ be any feasible solution for the primal  \eqref{maxent}. Then as $\int q(\x)d\x=1$,
 
 \begin{eqnarray*}
 \ln \int \exp(\la\blambda,\v_{2n}(\x))\,d\x&=&\ln \int \frac{\exp(\la\blambda,\v_{2n}(\x))}{q(\x)}\,q(\x)\,d\x\\
 &\geq&\int \ln\frac{\exp(\la\blambda,\v_{2n}(\x))}{q(\x)}q(\x)\,d\x\quad\mbox{[by Jensen's inequality]}\\
 &=&\int \la\blambda,\v_{2n}(\x)\ra q(\x)\,d\x-\int q(\x)\,\ln q(\x)\,d\x\\
 &=&c^*\la\blambda,\bmu^{(2n)}\ra -\int q(\x)\,\ln q(\x)\,d\x\,,
 \end{eqnarray*}
 which yields 
 $\int q(\x)\,\ln q(\x)\,d\x\geq \la\blambda,c^*\bmu^{(2n)}\ra -\Psi(\blambda)$. As 
 $\blambda$ and $q$ were arbitrary, one obtains
  the desired result 
  \[\inf_{0<q}\int q(\x)\,\ln q(\x)\,d\x\,\geq\,\sup_{\blambda}
   \la\blambda,c^*\bmu^{(2n)}\ra -\Psi(\blambda)\,=\,\Psi^*(c^*\bmu^{(2n)})\,.\]
 We next show that strong duality holds. Let $\x\mapsto q^*(\x):=c^*\exp(-g(\x))$, which by construction is admissible for \eqref{maxent}, and recall  $1/c^*=\int\exp(-g(\x))\,d\x$.  
 Then 
 \begin{eqnarray*}
 \int q^*(\x)\,\ln q^*(\x)\,d\x&=&\ln c^*\,\underbrace{\int q^*d\x}_{=1}-c^*\int g(\x)\exp(-g(\x))\,d\x\\
 &=&\ln c^*-c^*\,c_{2n}\int \exp(-g(\x))\,d\x\\
 &=&\ln c^*-c_{2n}\,.
 \end{eqnarray*}
On the other hand with $\blambda^*=-\g$, one obtains
\begin{eqnarray*}
 \la-\g,c^*\bmu^{(2n)}\ra -\Psi(-\g)&=&-c^*\int g(\x)\exp(-g(\x))\,d\x-\ln\int \exp(-g(\x))\,d\x\\
 &=&-c^*\,\frac{c_{2n}}{c^*} +\ln c^*\,=\,\ln c^*-c_{2n}\,,
 \end{eqnarray*}
which implies that $q^*$ (resp. $\blambda^*$) is an optimal solution of \eqref{maxent} 
(resp.  \eqref{maxent-dual}) as both are feasible with same value. Uniqueness follows from the strict convexity 
of  $\int q\ln q\,d\x$ and $\Psi(\blambda)$.
 \end{proof}
 
 \subsection{Link with Ward's identities}
  Recall that the identities \eqref{aux-0} can be deduced from
 Stokes theorem applied to integration w.r.t. $d\x$ on $G$; see Remark \ref{rem-stokes}. But
 \eqref{aux-0} is also related to Ward's identities derived with an infinitesimal  rescaling argument. 
 
 Consider the infinitesimal scaling
 $\x\mapsto \x':=(1+\varepsilon)\x$ with $\varepsilon\ll1$. Under this change of variables,
 \[d\x'\,=\,(1+d\,\varepsilon)d\x+O(\varepsilon^2)\]
 Similarly, by homogeneity,
 \[g(\x')\,=\,g((1+\varepsilon)\x)\,=\,g(\x)+2n\,\varepsilon\,g(\x)+O(\varepsilon^2)\,,\]
 and therefore,  the Boltzmann factor becomes
 \[\exp(-g(\x'))\,=\,\exp(-g(\x))(1-2n\,\varepsilon\,g(\x))+O(\varepsilon^2)\]
 By invariance of $Z(g)$ under the change of variables, expansion to first order in $\varepsilon$ yields
 \[\int_{\R^d}(d-2n\,g(\x))\,\exp(-g(\x)))\,d\x\,=\,0\,\]
 which yields the first Ward's identity
 \[\frac{\int_{\R^d} g(\x)\,d\mu}{Z(g)}\,=\,\frac{d}{2n}\,.\]
 Next, for an arbitratry smooth $F(\x)$, proceding in a similar manner,
 \[F(\x')\,=\,F((1+\varepsilon)\,\x)\,=\,F(\x)+\varepsilon\,\la \x,\nabla F(\x)\ra+O(\varepsilon^2)\]
 and expanding the integral to first order yields
 \[\int_{\R^d}[\la\x,\nabla F(\x)\ra+(d-2n\,g(\x))\,F(\x)]\,\exp(-g(\x))\,d\x\,=\,0\,.\]
 Specializing to $F(\x)=\x^\balpha$ with $\balpha\in\N^d$, yields
 \[\int_{\R^d}\x^\balpha\,g(\x)\,\exp(-g(\x))\,d\x\,=\,
 \frac{d+\vert\balpha\vert}{2n}\,\int_{\R^d}\x^\balpha\,\exp(-g(\x))\,d\x\,,\]
 which is exactly \eqref{aux-1}.
 \section{A numerical scheme}
   Recall that $Z(g)=\Gamma(1+(d+2n)/(2n))\mathrm{vol}(G)$ with $G:=\{\x: g(\x)\leq 1\}$. 
  Then next, one describes a numerical scheme to approximate as closely as desired 
  $\mathrm{vol}(G)$ (and hence $Z(g)$ as well). 
  
  Let $\B\subset\R^d$ be the Euclidean unit ball (or the unit box $[-1,1]^d$),
  and let $\N^d_{\leq n}:=\{\balpha\in\N^d:\vert\balpha\vert\leq n\}$. As $0<g$, its level set $G$ is compact and 
  we may and will assume that $G\subset \B$ (possibly after rescaling). Let
  $\lambda$ be the Lebesgue measure restricted to $\B(0,1)$. Let
  $\mathscr{M}(\B)_+$ the space of finite Borel measures on $\B$. Then 
  \begin{equation}
  \label{LP-primal}
  \mathrm{vol}(G)\,=\,\max_{\phi\in\mathscr{M}(\B)_+} \{\,\phi(1):\:\phi\leq \lambda\,\}\,,\end{equation}
  i.e., $\mathrm{vol}(G)$ is the optimal value of an infinite-dimensional linear program, with  
  $\phi^*(d\x)=1_G(\x)\lambda(d\x)$ as optimal solution. A  dual of \eqref{LP-primal} reads
  \begin{equation}
  \label{LP-dual}
  \inf_{0\leq p\in \R[\x]}\,\{\, \int_{\B} p\,d\lambda:\: p\geq 1_G\,\}\,=\,\inf_{0\leq p\in\R[\x]}
  \{\,\Vert p\Vert_{L^1(\B)}: \:p\geq 1_G\,\}\,.\end{equation}
 One way to approximate $\mathrm{vol}(G)$ is to consider \eqref{LP-primal} as Generalized Moment Problem (GMP)
 and apply the Moment-SOS hierarchy described in e.g. \cite{acta}. It consists in solving a hierarchy of (finite-dimensional) semidefinite \emph{relaxations} of the infinite-dimensional LP \eqref{LP-primal} 
 whose associated sequence of optimal values is monotone non increasing and converges to $\mathrm{vol}(G)$. 
 At step $n$ of the hierarchy, one solves  
 \begin{equation}
 \label{sdp-t}
 \rho_t\,=\,\max_{\bphi}\,\{\,\phi_0:\: 0\,\preceq\,\M_t(\bphi)\,\preceq\,\M_t(\lambda)\,\}\end{equation}
 where the unknown $\bphi=(\phi_\balpha)_{\balpha\in\N^d_{\leq 2t}}$ is 
 a vector of pseudo-moments (of degree at most $2t$) and $\M_t(\bphi)$ is the moment matrix associated with $\bphi$.
 Its rows and columns are indexed by $(\x^\balpha)_{\balpha\in\N^d_{\leq t}}$, and with entries
 \[\M_t(\bphi)(\balpha,\bbeta)\,=\,\phi_{\balpha+\bbeta}\,,\quad\forall \balpha,\bbeta\in\N^d_{\leq t}\,.\]
 The optimization problem \eqref{sdp-t} is a \emph{semidefinite program} (SDP), a class of convex conic optimization problems that can be solved efficiently.  The dual of \eqref{sdp-t} is also a semidefinite program, which reads:
 \begin{equation}
 \label{sdp-t-dual}
 \rho^*_t\,=\,\max_{p\in\Sigma[\x]_t}\,\{\,\int_\b p\,d\lambda: p\geq 1_G\,\}
  \end{equation}
  where $\Sigma[\x]_t$ is the space of sum-of-squares (SoS) polynomials of degree at most $2t$. When
  $G$ has nonempty interior then $\rho_t=\rho^*_t$ for all integer $t$, and $\rho_t\downarrow \mathrm{vol}(G)$ as $t\to\infty$.
  For more details 
  on the moment hierarchy and its application to solve \eqref{LP-primal} the interested reader is referred to e.g.
  \cite{acta} and \cite{Review} respectively.
  
  From the formulation \eqref{sdp-t-dual} of the dual, the convergence 
  $\rho^*_t\to\mathrm{vol}(G)$ is expected to be slow because one minimizes $\Vert p-1_G\Vert_{L^1(\B)}$
 by constructing a converging sequence $(p_t)_{t\in\N}\subset\Sigma[\x]$ (where $p_t$
 is a degree-$2t$ SOS). It is well-known that approximating a discontinuous function by
  a polynomial is tedious and prone to the so-called Gibbs phenomenon (oscillations at discontinuity and boundary points
  of $\B$).
  
  Let $g(\x)=\sum_{\bbeta\in\N^d_{2n}}g_\bbeta\,\x^\bbeta$. As $\phi^*=1_G(\x)\,d\x$ is the optimal solution of 
  \eqref{LP-primal},  by \eqref{aux-0}, it satisfies the linear Stokes moments constraints:
 \[\sum_{\bbeta\in\N^d_{2n}}g_{\bbeta}\,\phi^*_{\balpha+\bbeta}\,=\,\frac{d+\vert\balpha\vert}{d+2n+\vert\balpha\vert}\phi^*_{\balpha}\,,\quad \forall \balpha\in\N\,;\]
 see \cite[Lemma 1]{recovery} and Remark \ref{rem-stokes}.
 Therefore, for every $t\geq n$, one may and will include the additional linear moment constraints 
  \[\sum_{\bbeta\in\N^d_{2n}}g_{\bbeta}\,\phi^*_{\balpha+\bbeta}\,=\,\frac{d+\vert\balpha\vert}{d+2n+\vert\balpha\vert}\phi^*_{\balpha}\,,\quad \forall \balpha\in\N^d_{\leq 2(t-n)}\,,\]
  in the semidefinite program \eqref{sdp-t}, to obtain the stronger semidefinite relaxations
  \begin{equation}
 \label{sdp-t-stokes}
 \begin{array}{rl}
 \delta_t\,=\,\displaystyle\max_{\bphi}&\{\,\phi_0:\: 0\,\preceq\,\M_t(\bphi)\,\preceq\,\M_t(\lambda)\,;\\
 &\displaystyle\sum_{\bbeta\in\N^d_{2n}}g_{\bbeta}\,\phi^*_{\balpha+\bbeta}\,=\,\frac{d+\vert\balpha\vert}{d+2n+\vert\balpha\vert}\phi^*_{\balpha}\,,\quad \forall \balpha\in\N^d_{\leq 2(t-n)}\,\}\,.
 \end{array}
 \end{equation}
 The set of feasible solutions $\bphi$ of \eqref{sdp-t-stokes} is strictly contained in that of \eqref{sdp-t} and 
 it turns out that the convergence $\delta_t\downarrow \mathrm{vol}(G)$ as $t\to\infty$, is much faster than
 the convergence $\rho_t\downarrow \mathrm{vol}(G)$. 
  One rationale behind 
 this acceleration has been provided in \cite{stokes} with some numerical illustrative examples
 
\section{Conclusion}
 In this work we have investigated integral discriminants $Z(g)=\int \exp(-g)d\x$
 where  $g$ is a nonnegative form. These integrals generalize Gaussian partition functions and arise naturally in several contexts, including statistical mechanics, quantum field theory, and exponential families.
Our first main result is the identification of a Gaussian-like fixed-point relation between the coefficients of 
$g$ and the degree-$2n$ moments of the associated Boltzmann measure 
$d\mu=\exp(-g)d\x$. This relation extends the familiar correspondence between quadratic forms and covariances for Gaussian measures to arbitrary homogeneous polynomial actions. It provides an intrinsic characterization of the action in terms of its own moments and yields a natural higher-degree analogue of the Gaussian covariance identity.
We then established several variational characterizations of the polynomial action. In particular, we showed that 
$g$ can be recovered as the solution of variational problems involving  
$L^2(\mu)$-approximation and constrained minimization among homogeneous forms of fixed degree. Moreover, the normalized density proportional to $\exp(-g)$
is shown to solve a maximum-entropy problem under suitably scaled moment constraints, thereby extending the classical extremality of Gaussian distributions to the setting of higher-degree homogeneous actions.

These results provide a unified perspective on integral discriminants, connecting their algebraic structure, their moment representation, and their variational properties. They clarify in which sense homogeneous polynomial actions behave as higher-order analogues of quadratic Gaussian actions and make precise the role played by scale invariance and homogeneity in this correspondence.

From a computational viewpoint, the fixed-point relations and the associated moment identities 
have supplied natural additional linear moment constraints, which when incorporated  to 
a numerical scheme based on truncated moment hierarchies, dramatically accelerate its convergence. This opens the way to practical approximation of partition functions and correlators for non-quadratic actions using convex optimization techniques.

An natural direction for future work would be to extend the present framework beyond homogeneous polynomials, for instance to sums of homogeneous components or to non-polynomial convex actions. 
Overall, the fixed-point and variational viewpoint developed here provides a conceptual and technical bridge between Gaussian theory and more general polynomial actions, and may serve as a useful tool for both analytical and numerical investigations of non-Gaussian partition functions.
 
\end{document}